\documentclass[a4paper,11pt,reqno]{amsart}
\usepackage{ifthen}
\usepackage{graphicx}
\usepackage{amsfonts,amsmath,latexsym,amssymb,amsthm}
\newtheorem{theorem}{Theorem}[]
\newtheorem{corollary}{Corollary}[]

\newtheorem{lemma}{Lemma}[]
\newtheorem{prop}{Proposition}[]

\begin{document}
\bibliographystyle{amsplain}

\title[Stable Functions of Janowski Type]
{Stable Functions of Janowski Type}

\author[K. Chandrasekran]{Koneri Chandrasekran}
\address{Koneri Chandrasekran \\ Department of Mathematics \\ Jeppiaar SRR Engineering College, Affiliated to Anna University \\ Chennai 603 103, India}
\email{kchandru2014@gmail.com}

\author[D. J. Prabhakaran ]{Devasir John Prabhakaran}
\address{Devasir John Prabhakaran \\ Department of Mathematics \\ MIT Campus, Anna University \\ Chennai 600 044, India}
\email{asirprabha@yahoo.com}

\author[P. Sangal]{Priyanka Sangal}
\address{Priyanka Sangal\\ Department of Mathematics\\ R.C.U. Govt. P.G. College Uttarkashi\\ Uttarakhand 249193, India}
\email{sangal.priyanka@gmail.com}


\begin{abstract}
 A function $f\in \mathcal{A}_1$ is said to be stable
with respect to $g\in \mathcal{A}_1 $ if
\begin{align*}
\frac{s_n(f(z))}{f(z)} \prec \frac{1}{g(z)}, \qquad z\in\mathbb{D},
\end{align*}
holds for all $n \in \mathbb{N}$ where $\mathcal{A}_1$ denote the class of analytic functions $f$ in the unit disk
$\mathbb{D} =\{z\in \mathbb{C}: |z|<1 \}$ normalized by $f(0)=1$. Here $s_n(f(z))$, the $n^{th}$ partial sum
 of $f(z)=\displaystyle\sum_{k=0}^{\infty} a_kz^k$ is given by $s_n(f(z)) = \displaystyle\sum_{k=0}^{n} a_kz^k,
\ n\in \mathbb{N} \cup \{0\}$. In this work, we consider the following function
\begin{align*}
v_{\lambda}(A,B,z)=\left(\frac{1+Az}{1+Bz}\right)^{\lambda}
\end{align*}
for $-1\leq B  < A \leq 1$ and $0\leq \lambda \leq 1 $ for our investigation.
The main purpose of this paper is to prove that $v_{\lambda}(A,B,z)$ is stable with respect to
$\displaystyle v_{\lambda}(0,B,z)= \frac{1}{(1+Bz)^{\lambda}}$ for
$0 < \lambda \leq 1 $ and $-1\leq B < A \leq 0$. Further, we prove that $v_{\lambda}(A,B,z)$
is not stable with respect to itself, when $0 < \lambda \leq 1 $ and $-1\leq B < A <0$.
\end{abstract}
\maketitle

\section*{Introduction \& Main Results}
Let $\mathcal{A}$ denote the family of functions $f$ that are analytic in the unit disk $ \mathbb{D}:=\{z:\, |z|<1\}$.
Let $\mathcal{A}_1$ is the subset of $\mathcal{A}$ with the normalization $f(0)=1$.
A single valued function $f\in\mathcal{A}_1$ is said to be univalent in a
domain $\Delta \subseteq \mathbb{C}$ if $f$ is one-to-one in $\Delta$.
The class of all univalent functions
with the normalization $f(0)=0=f'(0)-1$ is denoted by $\mathcal{S}$.
Let $\Omega $ be the family of functions $\omega$, regular in $\mathbb{D}$ and
satisfying the conditions $\omega(0)=0$ and $|\omega(z)|<1$ for all $z\in\mathbb{D}$.
For $f, g \in \mathcal{A}$, the function $f$ is said to be subordinate to $g$, denoted by $f \prec g$
if and only if there exists an analytic function $\omega\in\Omega$
such that $f=g\circ \omega$.
In particular, if $g$ is univalent in $\mathbb{D}$ then $f(0)=g(0)$
and $f(\mathbb{D}) \subseteq g(\mathbb{D})$ hold.

The function $zf(z)\in\mathcal{A}_1$ is
starlike of order $\lambda$ if ${\rm Re}\left(\dfrac{zf'(z)}{f(z)}\right) > \lambda$
for all $z \in \mathbb{D}$ and $0\leq\lambda < 1$.
The class of all starlike functions, denoted by $\mathcal{S}^{\ast}(\lambda)$ is a subclass of $\mathcal{S}$.
The $n^{th}$ partial sum $s_n(f(z))$ of $ f(z)=\displaystyle\sum_{k=0}^{\infty} a_k z^k$ is given by
$ s_n(f(z)) = \displaystyle\sum_{k=0}^{n} a_k z^k, n = 0,1,2,\ldots $.
For more details about the univalent functions, its subclasses and subordination properties, we
refer \cite{duren-1983-book, janowski-1973-ann-polon, ruscheweyh-1982-book}.

The concept of stable functions was first introduced by Ruscheweyh and Salinas
\cite{ruscheweyh-salinas-2000-AnnMarie}, while discussing the class of starlike functions
of order $\lambda$, where $1/2\leq\lambda<1$. However, the class of starlike functions
of order $\lambda\in[1/2,1)$ is comparatively a much narrow class but it has many
interesting properties too. Ruscheweyh and Salinas \cite{ruscheweyh-salinas-2000-AnnMarie}
proved the following result.
\begin{theorem}\label{thm:ruschweyh-starlike-order-lambda}
\cite{ruscheweyh-salinas-2000-AnnMarie}
Let $\lambda\in[1/2,1)$ and $zf\in\mathcal{S}^{\ast}(\lambda)$, then
\begin{align*}
\frac{s_n(f,z)}{f(z)}\prec (1-z)^{\lambda}, \qquad n\in\mathbb{N}, z\in\mathbb{D}.
\end{align*}
\end{theorem}
Theorem \ref{thm:ruschweyh-starlike-order-lambda} has several applications in Gegenbauer
polynomial sums and motivated by Theorem \ref{thm:ruschweyh-starlike-order-lambda},
Ruscheweyh and Salinas \cite{ruscheweyh-salinas-2000-AnnMarie} introduced the concept
of Stable functions which is stated as follows.
For some $n \in \mathbb{N}$, a function $F$ is said to be  $n$-stable function
with respect to $G$ if
\begin{align*}
\frac{s_n(F(z))}{F(z)} \prec \frac{1}{G(z)},\quad  \mbox{for $F, G \in \mathcal{A}_1$ and $z\in\mathbb{D}$}.
\end{align*}
Moreover, the function $F$  is said to be stable with respect to $G$, if $F$ is
$n$-stable with respect to $G$ for every $n\in \mathbb{N}$.
Particularly, if the function $F$ is $n$-stable with respect to itself. Then for every $n\in\mathbb{N}$,
$F$ is stable. In the present context, for $-1\leq B<A\leq 1$, we define a function
\begin{align*}
v_{\lambda}(A,B, z):= \left(\frac{1+Az}{1+Bz}\right)^{\lambda} \quad
\mbox{ for $z\in\mathbb{D}$ and $\lambda\in(0,1]$}.
\end{align*}
For $\lambda=1/2$, Ruscheweyh and Salinas \cite{ruscheweyh-salinas-2004-stable-JMAA}
proved that $v_{1/2}(1,-1,z)$ is stable function with respect to itself.
The stability of $v_{1/2}(1,-1,z)$ is equivalent to the simultaneous
non-negativity of general class of sine and cosine sums given by Vietoris \cite{vietoris-1958},
the most celebrated theorem of positivity of trigonometric sums.
Ruscheweyh and Salinas \cite{ruscheweyh-salinas-2004-stable-JMAA}
conjectured that $v_{\lambda}(1,-1,z)$ is stable for $0<\lambda<1/2$.
Using computer algebra, for $\lambda=1/4$ it was shown in \cite{ruscheweyh-salinas-2004-stable-JMAA}
that $v_{1/4}(1,-1,z)$ is $n$-stable for $n=1,2,3,\cdots,5000$. In the limiting case,
the validation of stability of $v_{\lambda}(1,-1,z)$ for $0<\lambda<1/2$ interpreted in terms of
positivity of trigonometric polynomials.

Further extensions of Vietoris Theorem and stable functions to Ces\`aro stable functions
and Generalized Ces\`aro stable functions have been
studied in \cite{saiful-2012-stable-results-in-math} and \cite{sangal-swami-stable-MIA} respectively.
In this direction, conjectures are also proposed in \cite{sangal-swami-stable-MIA} that
linked Generalized Ces\`aro stable functions with the positivity of trigonometric sums.
Chakraborty and Vasudevarao \cite{allu-stable-CMFT-2018} considered $A=1-2\alpha$, $B=-1$
and proved the following result.

\begin{theorem}
\label{thm:allu-stable}
\cite{allu-stable-CMFT-2018}
For  $0< \lambda \leq 1$ and  $1/2\leq \alpha<1$,
$ v_{\lambda}(1-2\alpha,-1,z)=\left(\dfrac{1+(1-2\alpha)z}{1-z}\right)^{\lambda}$
is stable with respect to $ v_{\lambda}(0,-1,z)=\dfrac{1}{(1-z)^{\lambda}}$.
\end{theorem}
Chakraborty and Vasudevarao \cite{allu-stable-CMFT-2018} also proved that
$ v_{\lambda}(1-2\alpha,-1,z)$ is not stable with repsect to itself when
$1/2<\alpha<1$ and $0<\lambda\leq1$. For $\lambda=1$, the function
$ v_{1}(A,B,z)=\dfrac{1+Az}{1+Bz}$ have been studied widely by many researchers.
The analytic functions of $\mathcal{A}_1$ subordinate
to $\dfrac{1+Az}{1+Bz}$ have been studied by Janowski \cite{janowski-1973-ann-polon} and
the class of such functions is denoted by $\mathcal{P}(A,B)$. The functions of
$\mathcal{P}(A,B)$ are called Janowski functions. Moreover, the set of functions
$zf\in\mathcal{A}_1$, for which $\dfrac{zf'(z)}{f(z)}\prec \dfrac{1+Az}{1+Bz}$ holds, called Janowski
starlike functions and the class of such functions is denoted by $\mathcal{S}^{\ast}(A,B)$. It can be easily seen that
$\mathcal{S}^{\ast}(1,-1)\equiv\mathcal{S}^{\ast}$.

In this paper, we show that  $v_{\lambda}(A,B,z)$ is stable with respect to
$v_{\lambda}(0,B,z)= 1/(1+Bz)^{\lambda}$ for $0 < \lambda \leq 1 $
and $-1\leq B < A\leq 0$. Further, $v_{\lambda}(A,B,z)$ is not stable with respect
to itself, when $0 < \lambda \leq 1 $ and $-1\leq B < A< 0$.
We can write $v_{\lambda}(A,B,z)$ as,
\begin{align}
v_{\lambda}(A,B,z) &= \left(\frac{1+Az}{1+Bz}\right)^{\lambda}\nonumber\\
                     &=\left({1+Az}\right)^{\lambda}\left({1+Bz}\right)^{-\lambda}\nonumber\\
                     &=\left(1+\sum_{k=1}^{\infty}\frac{[\lambda]_k}{k!}A^k z^k\right)\left(1+\sum_{k=1}^{\infty}\frac{(\lambda)_k}{k!}(-B)^k z^k\right)\nonumber\\
                     &=1+ \sum_{n=1}^{\infty}\left(\sum_{k=0}^{n}\frac{[\lambda]_k}{k!}\frac{(\lambda)_{n-k}}{(n-k)!}A^k(-B)^{n-k}\right)z^n,\label{eqn:define-v-lambda}
\end{align}
where $[\lambda]_k$ and $(\lambda)_k$ denote the factorial polynomials given as
\begin{align*}
 \left\{
   \begin{array}{ll}
     [\lambda]_k=& \lambda(\lambda-1)(\lambda-2)(\lambda-3)\cdots (\lambda-k+1), \quad  \mbox{ and } \\
     (\lambda)_k=& \frac{\Gamma (\lambda+k)}{\Gamma(\lambda)}
                                =\lambda(\lambda+1)\cdots (\lambda+k-1), \quad \mbox{ Where $\Gamma$ is well-known gamma function, }
   \end{array}
 \right.
\end{align*}
for $k=1,2,\cdots$ respectively with $[\lambda]_0 = 1 = (\lambda)_0$. So $v_{\lambda}(A,B,z)$
can be written as
\begin{align*}
 v_{\lambda}(A,B,z) = 1+ \sum_{n=1}^{\infty} a_n(A,B,\lambda) z^n,
\end{align*}
where
\begin{align*}
a_n:=a_n(A,B,\lambda) = \sum_{k=0}^{n}\frac{[\lambda]_k}{k!}\frac{(\lambda)_{n-k}}{(n-k)!}A^k(-B)^{n-k}.
\end{align*}
Now, we state two lemmas which will helpful to prove our main results.

\begin{lemma}\label{lemma:1}
For  $0<\lambda \leq 1$ and $-1 \leq B < A \leq 0$, we have
\begin{align*}
\sum_{k=0}^{n}\frac{[\lambda]_k}{k!}\frac{(\lambda)_{n-k}}{(n-k)!}A^k(-B)^{n-k} > 0.
\end{align*}
\end{lemma}

\begin{lemma}\label{lemma:2}
Let $v_{\lambda} (A,B,z)$ be defined by \eqref{eqn:define-v-lambda}. Then  for  $\lambda \in (0,1]$ and $-1 \leq B < A \leq 0$,
\begin{align}\label{eqn:lemma2-stat}
 (m+1)(n+1)\left(\sum_{k=0}^{n+1}\frac{[\lambda]_k}{k!}\frac{(\lambda)_{n+1-k}}{(n+1-k)!}A^kB^{n+1-k}\right)\nonumber\\ -mn \left(\sum_{k=0}^{n}\frac{[\lambda]_k}{k!}\frac{(\lambda)_{n-k}}{(n-k)!}A^kB^{n-k}\right) \geq 0
\end{align}
holds for all $m,\ n \in \mathbb{N}$.
\end{lemma}

Now, we state main results of this paper which are about the stability of $v_{\lambda}(A,B,z)$
with repsect to $v_{\lambda}(0,B,z)$ and $v_{\lambda}(A,B,z)$ itself.

\begin{theorem}\label{thm:janowski-stable-v}
For $\lambda \in (0,1]$ and $-1 \leq B <  A \leq 0$,  $v_{\lambda} (A,B,z)$ given in
\eqref{eqn:define-v-lambda} is stable with respect to
 $v_{\lambda} (0,B,z) = \dfrac{1}{(1+Bz)^{\lambda}}$.
\end{theorem}
If we substitute $A=0$ in Theorem \ref{thm:janowski-stable-v}, we get the following
corollary which is also a generalization of the result
given by Ruscheweyh and Salinas \cite{ruscheweyh-salinas-2000-AnnMarie}.
\begin{corollary}\label{cor:thm-A=0}
For $\lambda \in (0,1]$ and $-1 \leq B<0$,  $v_{\lambda} (0,B,z)= \dfrac{1}{(1+Bz)^{\lambda}}$
is stable function.
\end{corollary}
Now for $0<\mu\leq \lambda\leq 1$, we have the following corollary of Theorem \ref{thm:janowski-stable-v}.
%
\begin{corollary}\label{cor:thm-main}
For $0<\mu\leq \lambda \leq 1$ and for $-1\leq B<0$ we have
\begin{align*}
\frac{s_n(v_{\mu}(0,B,z))}{v_{\lambda}(0,B,z)} \prec \frac{1}{v_{\lambda}(0,B,z)}, \qquad \mbox{ for $z\in\mathbb{D}$}.
\end{align*}
\end{corollary}
Theorem \ref{thm:janowski-stable-v} also generalizes result of Chakraborty and Vasudevarao \cite{allu-stable-CMFT-2018} as
if we substitute $A=1-2\alpha$ and $B=-1$ in Theorem \ref{thm:janowski-stable-v},
reduces to Theorem \ref{thm:allu-stable}. In other words, Theorem \ref{thm:allu-stable} is a particular case of
Theorem \ref{thm:janowski-stable-v}.

\begin{theorem}\label{thm:not-janowshi-stable-v}
For $\lambda \in (0,1]$ and $-1 \leq B <  A < 0$,  $v_{\lambda} (A,B,z)=\left(\dfrac{1+Az}{1+Bz}\right)^{\lambda}$
is not stable with respect to itself.
\end{theorem}

\section*{Proof of Main Results}
\begin{proof}[Proof of Lemma \ref{lemma:1}]
Consider,
\begin{align*}
1&=(1-z)^{\lambda}(1-z)^{-\lambda}\\
&=1+\sum_{n=1}^{\infty}\left(\sum_{k=0}^{n}
\frac{[\lambda]_k}{k!}\frac{(\lambda)_{n-k}}{(n-k)!}(-1)^{k}\right)z^n
\end{align*}
Comparing the coefficients of $z^n$ on both the sides we have
\begin{align*}
\sum_{k=0}^{n}\frac{[\lambda]_k}{k!}\frac{(\lambda)_{n-k}}{(n-k)!}(-1)^{k}=0,
\end{align*}
which can be expanded as
\begin{align*}
& \frac{(\lambda)(\lambda+1)\cdots (\lambda+n-1)}{(n)!}
    +\frac{(\lambda)(\lambda+1)\cdots(\lambda+n-2)}{(n-1)!}\left(\frac{\lambda}{1!}\right) (-1)\\
     &\qquad +\frac{(\lambda)(\lambda+1)\cdots(\lambda+n-3)}{(n-2)!}\frac{\lambda(\lambda-1) }{2!} (-1)^2+\cdots +\frac{\lambda}{1!}\frac{(\lambda)(\lambda-1)\cdots(\lambda-n+2)}{(n-1)!}(-1)^{n-1}\\
    &\qquad \qquad+\frac{(\lambda)(\lambda-1)\cdots(\lambda-n+1)}{n!}(-1)^n =0.
\end{align*}
Since $0\leq \lambda<1$, so only first term in the above equation is positive. By multiplying
$2^{nd}, 3^{rd},\cdots$,  $(n+1)^{th}$ terms by $\dfrac{\alpha}{\beta}$, $\dfrac{\alpha^2}{\beta^2}$, $\cdots$,
$\dfrac{\alpha^n}{\beta^n} $ respectively, we obtain for $0\leq \alpha<\beta$,
\begin{align*}
& \frac{(\lambda)(\lambda+1)\cdots (\lambda+n-1)}{(n)!}
    +\frac{(\lambda)(\lambda+1)\cdots(\lambda+n-2)}{(n-1)!}\left(\frac{\lambda}{1!}\right) (-1)\frac{\alpha}{\beta}\\
     &\quad +\frac{(\lambda)(\lambda+1)\cdots(\lambda+n-3)}{(n-2)!}\frac{\lambda(\lambda-1) }{2!} \left(\frac{-\alpha}{\beta}\right)^2+\cdots +\frac{\lambda}{1!}\frac{(\lambda)(\lambda-1)\cdots(\lambda-n+2)}{(n-1)!}\left(\frac{-\alpha}{\beta}\right)^{n-1}\\
    &\qquad \qquad+\frac{(\lambda)(\lambda-1)\cdots(\lambda-n+1)}{n!}\left(\frac{-\alpha}{\beta}\right)^n \geq0.
\end{align*}
After multiplying by $\beta^n$ we obtain
\begin{align}\label{eqn:lemma1-alpha-beta}
\beta^n\sum_{k=0}^n \frac{[\lambda]_k}{k!}\frac{(\lambda)_{n-k}}{(n-k)!}(-1)^k
            \left(\frac{\alpha}{\beta}\right)^k\geq0
\end{align}
By substituting $\alpha=-A$, $\beta=-B$ in \eqref{eqn:lemma1-alpha-beta}
so that for $-1\leq B <A\leq 0$, the lemma is proved.
\end{proof}

\begin{proof}[Proof of Lemma \ref{lemma:2}]
Let $v_{\lambda}(A,B,z)$ be defined by \eqref{eqn:define-v-lambda}. Then,
\begin{align}
v_{\lambda} (A,B,z) &= \left(\frac{1+Az}{1+Bz}\right)^\lambda=1+a_1z+a_2z^2+a_3z^3+\cdots\nonumber\\
v_{\lambda}' (A,B,z) &= \lambda\left(\frac{1+Az}{1+Bz}\right)^{\lambda-1}
                                            \left(
                                            \frac{(1+Bz)A-(1+Az)B}{(1+Bz)^2}
                                            \right)\nonumber\\
                                            &=\lambda (A-B)\frac{(1+Az)^{\lambda-1}}{(1+Bz)^{\lambda+1}}\nonumber\\
(1+Bz)v_{\lambda}' (A,B,z) &=\lambda(A-B)(1+Az)^{\lambda-1}(1+Bz)^{-\lambda}\label{eqn:lemma2-der-v}
\end{align}
Since $0>A>B$, $0<\lambda\leq 1$,  $(1+Az)^{\lambda-1}=1+(\lambda-1)Az
+\frac{(\lambda-1)(\lambda-2)}{2!}A^2z^2+\cdots$ and $(1+Bz)^{-\lambda}=1-\lambda Bz
+\frac{\lambda(\lambda+1)}{2!}B^2z^2+\cdots$ have positive Taylor series coefficients.
A simple computation yields that
\begin{align}
(1+Bz)v_{\lambda}' (A,B,z)&=(a_1+2a_2z+3a_3z^2+\cdots)(1+Bz)\nonumber\\
                                    &=a_1+\sum_{n=1}^{\infty}((n+1)a_{n+1}+Bna_n)z^n. \label{eqn:lemma2-der-v-an}
\end{align}
Since right hand side of \eqref{eqn:lemma2-der-v} has positive Taylor coefficients, from
\eqref{eqn:lemma2-der-v} and \eqref{eqn:lemma2-der-v-an} we conclude that
\begin{align}\label{eqn:lemma2-coeff-ineq}
(n+1)a_{n+1}+Bna_n>0, \quad n\in\mathbb{N}.
\end{align}
The left hand side of the expression given in \eqref{eqn:lemma2-stat} can be rewritten as
\begin{align}\label{eqn:lemma2-mnan}
(m+1)(n+1)a_{n+1}+mnBa_n.
\end{align}
Equivalently, \eqref{eqn:lemma2-mnan} can be written as
\begin{align*}
m((n+1)a_{n+1}+Bna_n)+(n+1)a_{n+1}.
\end{align*}
Using \eqref{eqn:lemma2-coeff-ineq} and the fact that $a_n\geq0$ for $m,n\in\mathbb{N}$,
the lemma is proved for $\lambda \in (0,1]$ and $-1 \leq B < A \leq 0$.
\end{proof}

Before going to proceed further for the proof of Theorem \ref{thm:janowski-stable-v},
it is easy to verify the following relations.
\begin{equation}\label{eqn:relations}
\begin{split}
     s'_n(v_{\lambda}(A,B,z),z)& = s_{n-1}(v'_{\lambda}(A,B,z),z), \\
     zs'_n(v_{\lambda}(A,B,z),z)& =  s_{n}(zv'_{\lambda}(A,B,z),z), \\
     z^2s'_n(v_{\lambda}(A,B,z),z)& =  s_{n}(z^2v'_{\lambda}(A,B,z),z).
 \end{split}
 \end{equation}
Now, we are ready to give the proof of Theorem \ref{thm:janowski-stable-v}.
\begin{proof}[Proof of Theorem \ref{thm:janowski-stable-v}]
To show that $v_{\lambda} (A,B,z)$ is stable with respect to $v_{\lambda}(0,B,z)$,
it is enough to show that
\begin{align*}
\frac{s_n(v_{\lambda}(A,B,z),z)}{v_{\lambda}(A,B,z),z} \prec \frac{1}{v_{\lambda}(0,B,z)}, \qquad z\in\mathbb{D}
\end{align*}
for all $n \in \mathbb{N}$, i.e.,  to prove that
\begin{align*}
\frac{(1+Bz)^{\lambda}s_n(v_{\lambda}(A,B,z),z)}{(1+Az)^{\lambda}} \prec (1+Bz)^{\lambda}, \qquad z\in\mathbb{D},
\end{align*}
which can be equivalently written as
\begin{align*}
\frac{(1+Bz)s_n(v_{\lambda}(A,B,z),z)^{\frac{1}{\lambda}}}{(1+Az)} \prec (1+Bz).
\end{align*}
To show that  , it is enough to prove that
\begin{align*}
\left| \displaystyle \frac{(1+Bz)s_n(v_{\lambda}(A,B,z),z)^{\frac{1}{\lambda}}}{(1+Az)}-1 \right|\leq |B|\leq 1, \quad z \in \mathbb{D}.
\end{align*}
For fixed $n$ and $\lambda$, we consider the following function
\begin{align*}
h(z)=1-\displaystyle \frac{(1+Bz)s_n(v_{\lambda}(A,B,z),z)^{\frac{1}{\lambda}}}{(1+Az)}, \quad z \in \mathbb{D}.
\end{align*}
It is easy to see that
\begin{align*}
 v'_{\lambda}(A,B,z) = \lambda (A-B)\frac{(1+Az)^{\lambda-1}}{(1+Bz)^{\lambda+1}}
 =\lambda (A-B)\frac{v_{\lambda}(A,B,z)}{(1+Bz)(1+Az)},
\end{align*}
which can be rewritten in the following form
\begin{align}\label{eqn:proof-main-thm-zero}
v_{\lambda}(A,B,z) - \frac{(1+(A+B)z+ABz^2)}{\lambda(A-B)}v'_{\lambda}(A,B,z)=0
\quad  \mbox{for $\ z \in \mathbb{D}$}.
\end{align}
A simple calculations gives that
\begin{align}
h'(z)&=\frac{A-B}{(1+Az)^2}s_n(v_{\lambda}(A,B,z),z)^{1/\lambda}
            -\frac{(1+Bz)}{(1+Az)\lambda}s_n(v_{\lambda}(A,B,z),z)^{\frac{1}{\lambda}-1}s'_n(v_{\lambda}(A,B,z),z)\nonumber\\
&= \frac{(A-B)s_n(v_{\lambda}(A,B,z),z)^{\frac{1}{\lambda}}}{(1+Az)^2}
            \left( s_n(v_{\lambda}(A,B,z),z)- \frac{(1+Az)(1+Bz)}{(A-B)\lambda}s'_n(v_{\lambda}(A,B,z),z)\right)\label{eqn:proof-main-thm-h'}
\end{align}
Using relations \eqref{eqn:relations} in \eqref{eqn:proof-main-thm-h'}, we get
\begin{equation}\label{eqn:proof-main-thm-h'-sn}
\begin{split}
h'(z)=&\frac{(A-B)s_n(v_{\lambda}(A,B,z),z)^{\frac{1}{\lambda}-1}}{(1+Az)^2}
            \bigg[
                s_n\left(v_{\lambda}(A,B,z)-\frac{(1+Az)(1+Bz)}{(A-B)\lambda}v'_{\lambda}(A,B,z),z\right)\\
                &\qquad \qquad\qquad+\frac{(n+1)}{\lambda (A-B)}\sum_{k=0}^{n+1}\frac{[\lambda]_k}{k!}
                        \frac{(\lambda)_{n-k+1}}{(n-k+1)!}A^k (-B)^{n-k+1}z^n\\
               &\qquad \qquad\qquad\qquad  -\frac{nAB}{\lambda(A-B)}\sum_{k=0}^{n}\frac{[\lambda]_k}{k!}
                        \frac{(\lambda)_{n-k}}{(n-k)!}A^k (-B)^{n-k}z^{n+1}\bigg].
\end{split}
\end{equation}
Substituting \eqref{eqn:proof-main-thm-zero} in \eqref{eqn:proof-main-thm-h'-sn} and using
definition of $a_n$, the following form of $h'(z)$ can be obtained.
\begin{align*}
h'(z)=&\frac{z^ns_n(v_{\lambda}(A,B,z),z)^{\frac{1}{\lambda}-1}}{\lambda(1+Az)^2}
            \left( (n+1)a_{n+1}-ABzna_n\right)\\
            =&\frac{z^ns_n(v_{\lambda}(A,B,z),z)^{\frac{1}{\lambda}-1}}{\lambda}
            \left( (n+1)a_{n+1}-ABzna_n\right)(1+Az)^{-2}\\
            =&\frac{z^ns_n(v_{\lambda}(A,B,z),z)^{\frac{1}{\lambda}-1}}{\lambda}
            \left( (n+1)a_{n+1}-ABzna_n\right)(1-2Az+3A^2z^2-4A^3z^3+\cdots)\\
            =&\frac{z^ns_n(v_{\lambda}(A,B,z),z)^{\frac{1}{\lambda}-1}}{\lambda}
            \left( (n+1)a_{n+1}+\sum_{m=1}^{\infty}(m+1)(n+1)a_{n+1}+mnBa_n(-A)^mz^m\right).
\end{align*}
Since $A \in (-1,0]$, we have $-A \geq 0$. Therefore in view of Lemma \ref{lemma:1},
we obtain $a_n  > 0$ for all $n\in \mathbb{N}$. Further, from Lemma \ref{lemma:2},
we obtain $(m+1)(n+1)a_{n+1}+Bmna_n  > 0$ for all $n,m \in \mathbb{N}$.
Thus
\begin{align*}
(n+1)a_{n+1}+\sum_{m=0}^{\infty}[(m+1)(n+1)a_{n+1}+Bmna_n](-A)^mz^m
\end{align*}
represents a series of positive Taylor's coefficients. Since $a_n > 0$ for all $n\in \mathbb{N}$,
the function $v_{\lambda}(A,B,z)$ has a series representation with positive Taylor coefficients.
Hence,
$$
|s_n(v_{\lambda}(A,B,z),z)|\leq s_n(v_{\lambda}(A,B,z),|z|)$$
holds and consequently
\begin{align}\label{eqn:proof-main-thm-mod-h'}
|h'(z)|\leq h'(|z|),\quad \mbox{ for all $ z\in\mathbb{D}$ holds.}
\end{align}
Since $h(0)=0$ and $h(-B)=1$, using \eqref{eqn:proof-main-thm-mod-h'} we obtain
\begin{align*}
|h(z)|=\left|\int_{0}^{z}h'(t)dt\right|\leq \int_{0}^{-B}h'\left(\frac{-tz}{B}\right)dt\leq \int_{0}^{-B}h'(t)dt=1,\quad  \ z \in \mathbb{D}.
\end{align*}
Therefore,
\begin{align*}
\left|\frac{(1+Bz)(s_n(v_(A,B,z),z))^{\frac{1}{\lambda}}}{(1+Az)}-1\right|<1,\quad  \ z \in \mathbb{D}.
\end{align*}
which implies that
\begin{align*}
\frac{s_n(v_{\lambda}(A,B,z),z)}{v_{\lambda}(A,B,z)} \prec \frac{1}{v_{\lambda}(0,B,z)}.
\end{align*}
Therefore, $v_{\lambda}(A,B,z)$ is $n$-stable with respect to $v_{\lambda}(0,B,z)$ for all
$n \in \mathbb{N}$. Hence $v_{\lambda}(A,B,z)$
is stable with respect to $v_{\lambda}(0,B,z)$ for all
$0< \lambda \leq 1$ and $-1\leq B < A \leq 0$.
\end{proof}

For the proof of Corollary \ref{cor:thm-main}, we need the following proposition which follows the same
procedure as given in \cite{koumandos-ruscheweyh-2007-conjecture-JAT}.
\begin{prop}
Let $\alpha,\beta>0$ and $B\in[-1,0)$. If $F\prec (1+Bz)^{\alpha}$ and $G\prec (1+Bz)^{\beta}$ then
$FG\prec (1+Bz)^{\alpha+\beta}$ for $z\in\mathbb{D}$.
\begin{proof}
The function $\log(1+Bz)$ is convex univalent for $z\in\mathbb{D}$ and $B\in[-1,0)$.
Our claim follows from
\begin{align*}
\frac{1}{\alpha+\beta}\log(F(z)G(z))&=\frac{\alpha}{\alpha+\beta}\log(1+Bu(z))
                                                                +\frac{\beta}{\alpha+\beta}\log(1+Bv(z))\\
                                                                &\prec \log(1+Bz),
\end{align*}
where $u,v$ are analytic functions such that $|u(z)|\leq |z|$ and $|v(z)|\leq |z|$ for $z\in\mathbb{D}$.
\end{proof}
\end{prop}
Now we are ready to give proof of Corollary \ref{cor:thm-main}.
\begin{proof}[Proof of Corollory \ref{cor:thm-main}]
For $0<\mu\leq \lambda\leq 1$ we have,
\begin{align*}
&\frac{1}{\lambda} \log\left((1+Bz)^{\lambda}s_n(v_{\mu}(0,B,z),z)\right)\\
&=\frac{1}{\lambda} \log\left[(1+Bz)^{\lambda-\mu}(1+Bz)^{\mu}s_n(v_{\mu}(0,B,z),z)\right]\\
&=\frac{1}{\lambda} \log(1+Bz)^{\lambda-\mu}+\frac{1}{\lambda} \log\left[(1+Bz)^{\mu}s_n(v_{\mu}(0,B,z),z)\right]\\
&=\frac{1}{\lambda} \log(1+Bu(z))^{\lambda-\mu}+\frac{1}{\lambda} \log(1+Bw(z))^{\mu}\\
&\prec (1+Bz)^{\lambda}
\end{align*}
for $|u(z)|\leq |z|$ and $|w(z)|\leq |z|$. Therefore, $(1+Bz)^{\lambda}s_n(v_{\mu}(0,B,z),z)\prec (1+Bz)^{\lambda}$
holds for all $z\in\mathbb{D}$ and $0<\mu\leq \lambda\leq 1$.
\end{proof}

Now we prove that $v_{\lambda}(A,B,z)$ is not stable with respect to
itself for $\lambda \in(0,1]$ and $ -1 \leq B < A \leq 0$.

\begin{proof}[Proof of Theorem \ref{thm:not-janowshi-stable-v}]
For $-1\leq B<A\leq 0$, to prove that $v_{\lambda}(A,B,z)$ is stable with respect to itself, we need to show that
\begin{align}\label{eqn:thm-not-stable-stat}
\frac{s_n(v_{\lambda}(A,B,z),z)}{v_{\lambda}(A,B,z)}\prec \frac{1}{v_{\lambda}(A,B,z)}, \quad z\in\mathbb{D}.
\end{align}
Equivalently $G(z)\prec H(z)$ where
\begin{align*}
G(z):= \frac{(1+Bz)s_n(v_{\lambda}(A,B,z),z)^{\frac{1}{\lambda}}}{1+Az}
\quad \mbox{and} \quad
H(z):=\frac{1+Bz}{1+Az}
\end{align*}
Since $G(z)$ and $H(z)$ are analytic in $\mathbb{D}$ for $ -1 \leq B < A \leq 0$
and $H(z)$ is univalent in $\mathbb{D}$. In the point of view of the subordination,
we have $G(z) \prec H(z)$ if and only if $G(0)=H(0)$ and $G(\mathbb{D}) \subseteq
H(\mathbb{D})$ and $G = H\circ \omega_1$, where $\omega_1\in\Omega $ analytic in $\mathbb{D}$
satisfying $\omega_1(0)=0$ and $|\omega_1(z)| < 1$.

In view of the Schwartz Lemma, we have $|\omega_1(z)|\leq |z|$ for $z \in \mathbb{D}$
and $|\omega_1'(0)|\leq 1$. If $G \prec H$, it follows that  $|G'(0)| \leq |H'(0)|$ and
$G(|z| \leq r) \subseteq H(|z|\leq r), 0\leq r <1$.

Let $\displaystyle \omega = H(z)=\frac{1+Bz}{1+Az}$, then
$\displaystyle z=\frac{1+B\omega}{1+A\omega}$. Therefore,
the image of $|z|\leq r$ under $H(z)$ is $\left|\dfrac{1+B\omega}{1+A\omega}\right|\leq r$ which
 after simplification is equivalent to $|w-C(r,A,B)|\leq R(r,A,B)$ where
\begin{align*}
  C(r,A,B):= \frac{r^2A-B}{B^2-r^2A^2} \quad \mbox{and} \quad  R(r,A,B):= \frac{r(A-B)}{B^2-r^2A^2}
\end{align*}

To show that $G\nprec H$, it is enough to show that $G(|z|\leq r) \nsubseteq H(|z|\leq r)$.
To prove that $G(|z|\leq r) \nsubseteq H(|z|\leq r)$, it is enough to choose a point
$z_0$ with $|z_0|\leq r_0$ such that $G(z_0)$ does not
lie in the disk $|\omega-C(r,A,B)|\leq R(r,A,B)$ for some $-1 \leq B < A \leq 0$.

\begin{figure}[h]
\includegraphics[width=\linewidth]{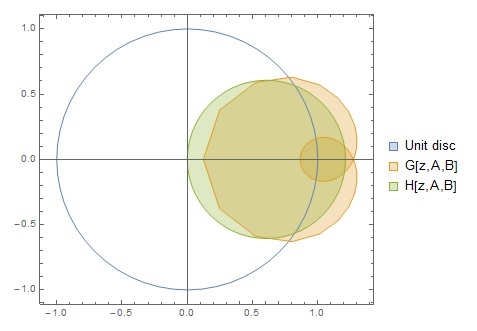}
\caption{$G(z_0,A,B)\nprec H(z_0,A,B)$ for $z_0=0.915282-0.357037 i$, $A =-0.679, B=-0.97,$ and $\lambda =0.3$.}
\label{fig:1}
\end{figure}
Choose $z_0 = 0.915282-0.357037 i$, $r_0 =0.98$, $A =-0.679, B=-0.97,$  $\lambda =0.3$
and $n = 1$. Then $G(z_0)=0.8697+0.5845 i$, $C(r_0,A,B) =0.634444 $ and
$R(r_0,A,B) =0.576521$. Clearly $G(z_0)$ does not lie in the disk
$|\omega-C(r_0,A,B)|\leq R(r_0,A,B) $. Therefore $G\nprec H$ i.e.,
\eqref{eqn:thm-not-stable-stat} does not hold. The graphical illustration of these values is
also given here in Figure \ref{fig:1}. Hence $v_{\lambda}(A,B,z)$ is not stable with respect to itself.
\end{proof}

\end{document}